\documentclass{amsart}
\usepackage[dvips]{graphicx}
\usepackage{url}

\usepackage{amsmath}
\usepackage{amsthm}
\usepackage{amssymb}
\usepackage{amsfonts}
\usepackage{mathrsfs}
\usepackage[all]{xy}
\usepackage{wrapfig}
\usepackage{graphicx}
\usepackage{float}
\usepackage{url}
\usepackage{subfig}

\newtheorem{theorem}{Theorem}[section]
\newtheorem{lemma}[theorem]{Lemma}

\newtheorem{proposition}[theorem]{Proposition}
\theoremstyle{definition}

\newtheorem{remark} [theorem] {Remark}

\newcommand{\vf}{\varphi}
\newcommand{\vp}{\varpi}

\newcommand{\G}{\Gamma}
\newcommand\res[1]{{\lower1pt\hbox{$|$}}_{\raise.5pt\hbox{${\scriptstyle #1}$}}}

\numberwithin{equation}{section}

\begin{document}

\title{A Wallis Product on Clovers}

\author{Trevor Hyde}
\address{Department of Mathematics, Amherst
College, Amherst, MA 01002-5000, USA}
\email{thyde641@gmail.com}


\date{\today}

\maketitle

\section{Introduction}
The Wallis product is a well-known infinite product expression for $\pi$ by rational factors derived by John Wallis in his 1655 treatise \cite{Wallis}. 
\begin{align}
\label{wallis}
	\pi &= 4\cdot \frac{2\cdot 4}{3\cdot 3}\cdot \frac{4\cdot 6}{5\cdot 5}\cdot \frac{6\cdot 8}{7\cdot 7}\cdot \frac{8\cdot 10}{9\cdot 9}\cdots \notag\\
	&= 4\prod_{n=1}^\infty{\frac{2n(2n+2)}{(2n+1)^2}}.
\end{align}
\indent One proof of \eqref{wallis} proceeds by considering the sequence of definite integrals
\[
	I(n) = \int_0^\pi{\sin(x)^n\,dx},
\]
and computing the limit of $I(n+1)/I(n)$ in two different ways. See \cite{stirling} for an excellent exposition of this argument. Let $\{\vp_m\}$ be the sequence of real numbers defined by
\begin{equation}
\label{vpdef}
	\vp_m = 2\int_0^1{\frac{dt}{\sqrt{1-t^m}}}.
\end{equation}
Observe that $\vp_2 = \pi$. The goal of this paper is to prove the generalized Wallis product formula
\begin{equation}
	\vp_m = \frac{2(m+2)}{m}\prod_{n=1}^\infty{\frac{2n(2mn+ m +2)}{(2n+1)(2mn+2)}}.
\end{equation}
In Section \ref{sec:clover} we recall the theory of \emph{clover curves} introduced in \cite{clover} where we realize $\vp_m$ as an arc length on the \emph{$m$-clover}. Our proof  generalizes the definite integral approach by considering the sequence of definite integrals
\[
	I_m(n) = \int_0^{\vp_m}{\vf_m(x)^n\,dx},
\]
where $\vf_m(x)$ is the \emph{$m$-clover function} to be defined in Section \ref{sec:clover}, and computing the limit of $I_m(n+1)/I_m(n)$ in two different ways.

\section{Clovers}
\label{sec:clover}

For natural $m$ we define the \emph{$m$-clover} to be the locus of the polar equation
\[
	r^{m/2} = \cos(\tfrac{m}{2}\theta).
\]
Examples of $m$-clovers are displayed in Figure 1 for small $m$. The $m$-clover has $m$ \emph{leaves} for $m$ odd and $\frac{m}{2}$ leaves for $m$ even. The \emph{principal leaf} is defined as the points on the $m$-clover satisfying $(r,\theta)\in [0,1]\times [-\tfrac{\pi}{m},\tfrac{\pi}{m}]$.
\begin{figure}[H]
\label{fig1}
	\centering
	\subfloat[Cardioid]{
		\includegraphics[scale=.35]{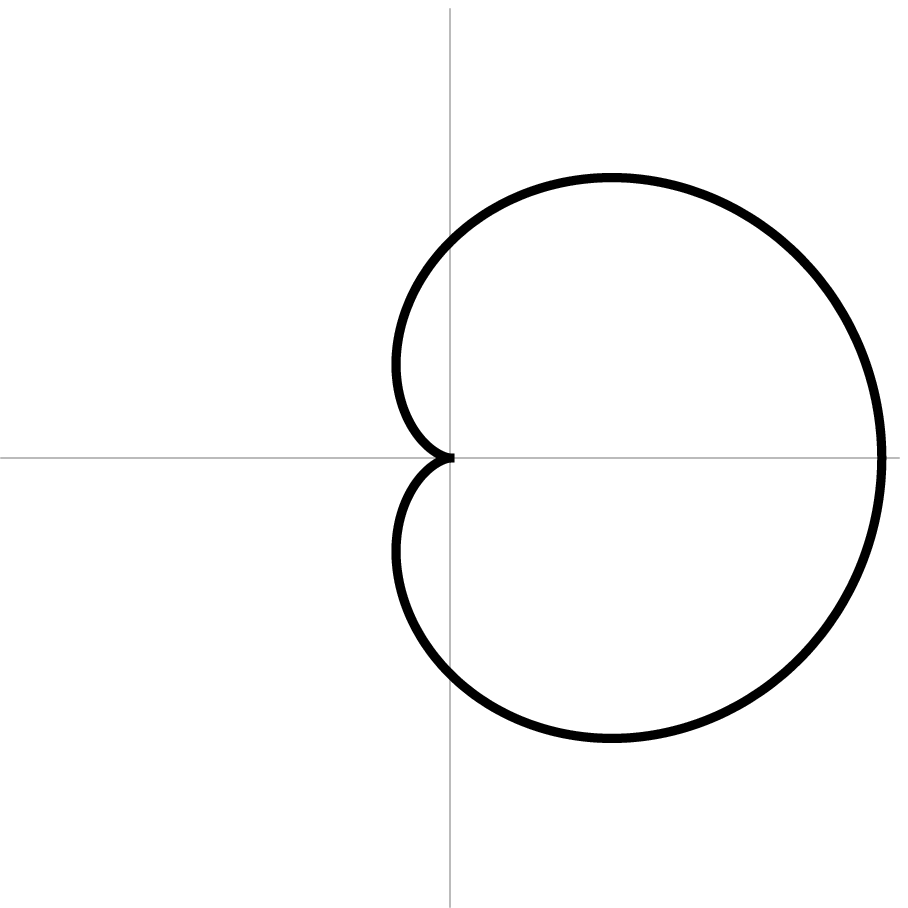}
	}
	\subfloat[Circle]{
		\includegraphics[scale=.35]{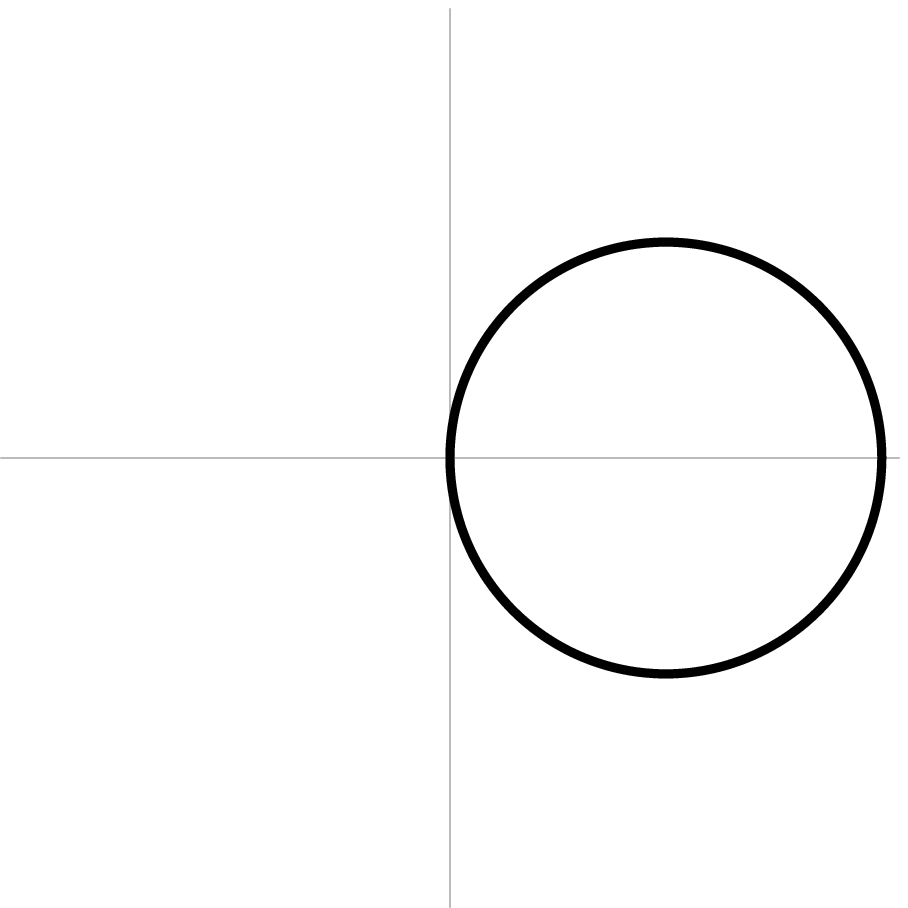}
	}
	\subfloat[Clover]{
		\includegraphics[scale=.35]{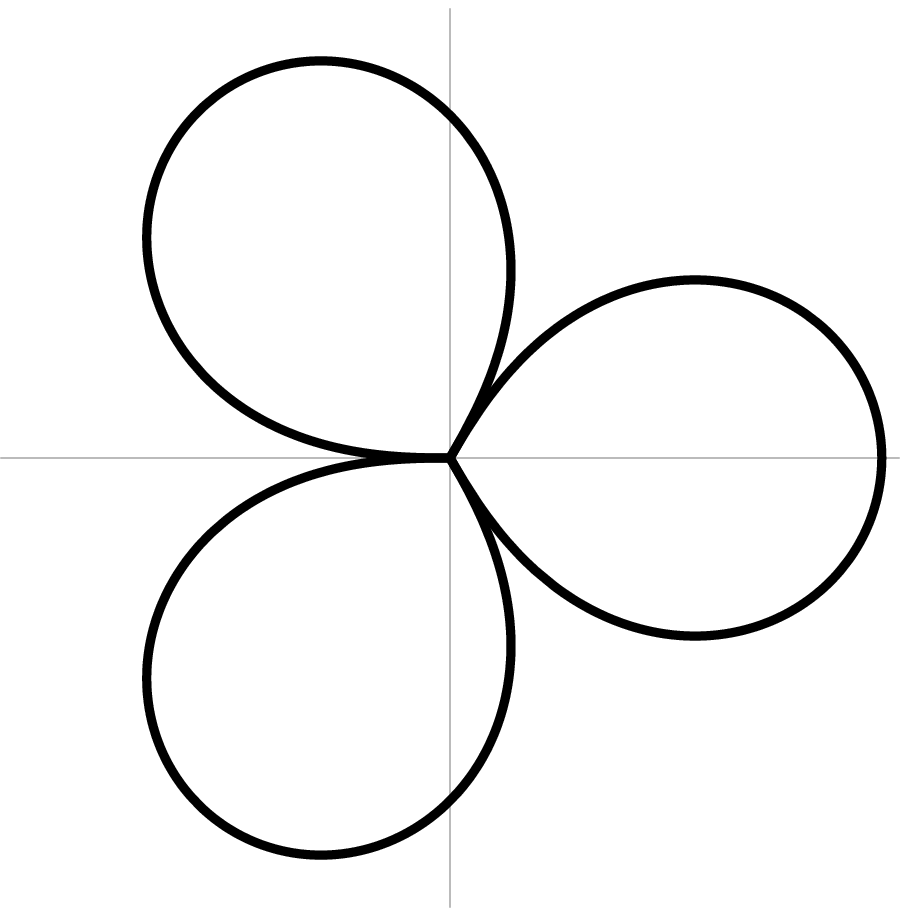}
	}
	\subfloat[Lemniscate]{
		\includegraphics[scale=.35]{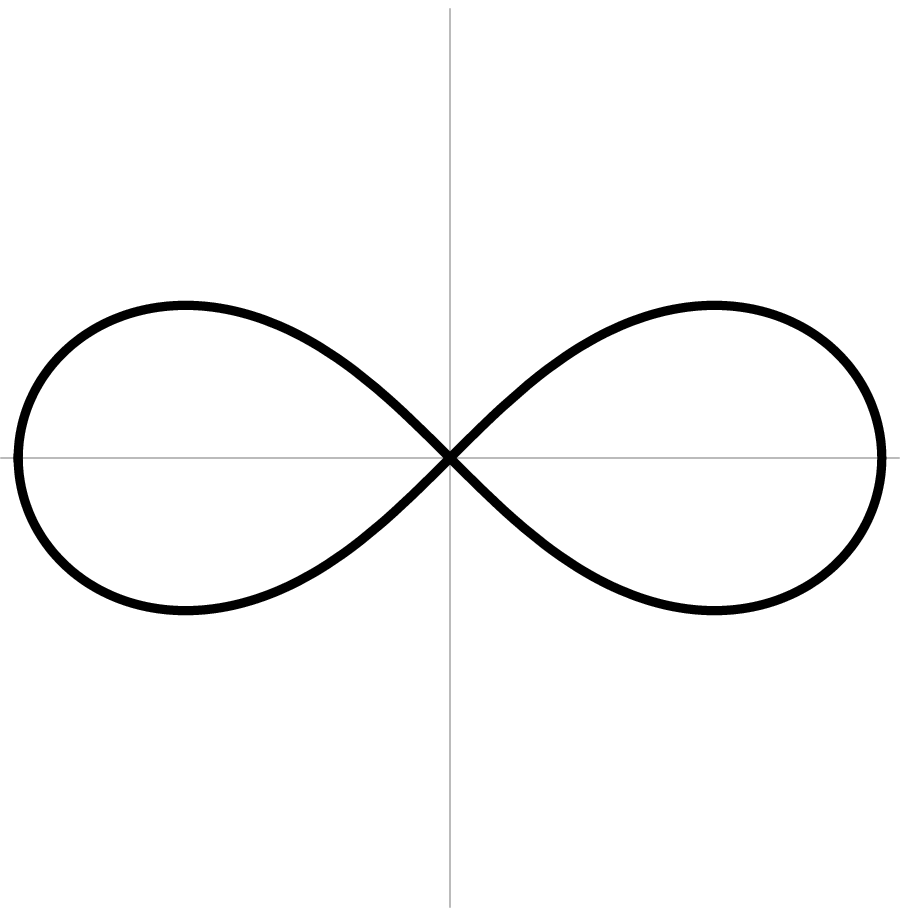}
	}
	\caption{$m$-clovers for $m=1,2,3,4$}
\end{figure}

Consider the polar arc length integral for a segment of the principal leaf in the upper-half plane beginning at the origin and terminating at the unique point with radial component $r\in [0,1]$,
\begin{equation}
\label{length}
	l_m(r) = \int_0^r{\frac{dt}{\sqrt{1-t^m}}}.\footnote{See \cite[Prop. 1]{clover} for a derivation of \eqref{length} from the definition of an arc length integral in polar coordinates.}
\end{equation}
Hence, the constant defined in \eqref{vpdef}
\begin{equation*}
	\vp_m = 2\int_0^1{\frac{dt}{\sqrt{1-t^m}}} = 2l_m(1)
\end{equation*}
denotes the arc length of the $m$-clover's principal leaf. The integrals \eqref{length} almost never have closed forms in elementary functions. One exception is the case $m=2$ where we have
\begin{equation}
	l_2(r) = \int_0^2{\frac{dt}{\sqrt{1-t^2}}} = \sin^{-1}(r).
\end{equation}
Thus, $l_2^{-1}(x) = \sin(x)$ is a natural function to consider. Motivated by this example we define the \emph{$m$-clover function} $\vf_m(x)$ for $x\in [0,\tfrac{1}{2}\vp_m]$ by
\begin{equation*}
	\vf_m(x) = r\mbox{ if }l_m(r) = x.
\end{equation*}
That is, define $\vf_m$ as the inverse of the arc length integral \eqref{length}.

\begin{figure}[H]
\centering
\includegraphics[scale=.7]{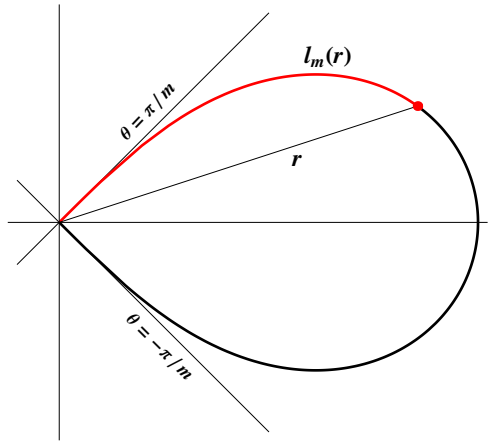}
\caption{Principal leaf of the $m$-clover.}
\end{figure}

We extend the domain of $\vf_m$ to $[0,\vp_m]$ via the functional equation
\[
	\vf_m(\vp_m - x) = \vf_m(x).
\]
Using one-sided derivatives at $0$ and $\vp_m$, $\vf_m(x)$ is differentiable on $[0,\vp_m]$ by the fundamental theorem of calculus. Geometrically, given $x\in[0,\vp_m]$, there is a unique point $\mathcal{P}_x = (r,\theta)$ at an arc distance $x$ along the principal leaf; $\vf_m(x) = r$ denotes the radial component of the point $\mathcal{P}_x$. In our proof of the generalized Wallis product, the function $\vf_m(x)$ plays the role of $\sin(x)$ in the proof of the original product in the sense that we shall consider the sequence of integrals
\begin{equation*}
	I_m(n) = \int_0^{\vp_m}{\vf_m(x)^n\,dx}.
\end{equation*}

\begin{remark} 
\label{remark}
For $m=2$, the $m$-clover is the circle $r=\cos(\theta)$ with $\vf_2(x) = \sin(x)$ and $\vp_2 = \pi$. The symbol $\vp$ is a variant on the Greek letter $\pi$. Thus, our use of this notation reflects the sense in which $\vp_m$ is a generalization of $\pi$. When $m=4$, the $m$-clover is the lemniscate $r^2 = \cos(2\theta)$ and  $\vf_m(x)$ is Abel's lemniscate function. The number 
\[
	\vp_4 = 2.6220575549\ldots
\]
is known as the \emph{lemniscate constant}. More information on the lemniscate, including its connections to number theory, may be found in \cite{agm}, \cite[Chp. 15]{galois}, and \cite{constants}.
\end{remark}

\begin{proposition}
\label{diffEq}
Let $\vf_m(x)$ be the $m$-clover function. Then for all $x\in [0,\vp_m]$
\begin{enumerate}
	\item $\vf_m(x)^m + \vf_m'(x)^2 = 1$, \text{ and}\\
	\item $\vf_m(x)^{m-1} = -\tfrac{2}{m}\vf_m''(x)$.
\end{enumerate}
\end{proposition}

\begin{proof}
	From the definition of $\vf_m(x)$ we have
	\begin{equation}
	\label{intIdentity}
		x = \int_0^{\vf_m(x)}{\frac{dt}{\sqrt{1-t^m}}},
	\end{equation}
	valid for all $x\in[0,\vp_m]$. Implicitly differentiating \eqref{intIdentity} gives
	\begin{align}
	\label{diff1}
		\frac{\vf_m'(x)}{\sqrt{1-\vf_m(x)^m}} &= 1\notag\\
		\vf_m'(x)^2 &= 1-\vf_m(x)^m .
	\end{align}
	Rearranging \eqref{diff1} results in (1),
	\begin{align}
	\label{diff2}
		\vf_m(x)^m + \vf_m'(x)^2 = 1.
	\end{align}
	Differentiating \eqref{diff2} yields (2)
	\begin{align}
		m\vf_m(x)^{m-1}\vf_m'(x) + 2\vf_m'(x)\vf_m''(x) &= 0 \notag\\
		\vf_m(x)^{m-1} &= - \tfrac{2}{m}\vf_m''(x).\qedhere
	\end{align}
\end{proof}

Proposition \ref{diffEq} (1) should be viewed as the $m$-clover version of the Pythagorean identity
\begin{equation}
\label{pyth}
	\sin(x)^2 + \cos(x)^2 = 1.
\end{equation}
In light of Remark \ref{remark}, \eqref{pyth} corresponds to the special case $m=2$. Analogies between $\vf_2(x) = \sin(x)$ and $\vf_4(x)$ have been explored since the introduction of $\vf_4(x)$ by Gauss and Abel. Abel's investigations produced his famous result characterizing the division points of the lemniscate constructible by ruler and compass in parallel with Gauss' characterization on the circle. Cox and Shurman extend the constructibility results to the $3$-clover in \cite{clover} where they introduce the $m$-clover theory outlined above.

\section{A Sequence of Integrals}

For a natural number $m$ and an integer $n\geq 0$, we define
\begin{align}
\label{sequence}
	I_m(n) = \int_0^{\vp_m}{\vf_m(x)^n\,dx}.
\end{align}
Each integral is finite and positive. Note that when $m=2$, $\{I_2(n)\}$ reduces to the sequence of definite integrals stated in the introduction. Our first task is to establish a recursive relation among the elements of $\{I_m(n)\}$.

\begin{lemma}
\label{recurLemma}
For all natural $m$ and integral $n\geq 0$,
\begin{align}
\label{intRecur}
	\frac{I_m(n+m)}{I_m(n)} = \frac{2(n+1)}{2(n+1)+ m}.
\end{align}
\end{lemma}

\begin{proof}
	Our strategy is to transform $I_m(n+m)$ with integration by parts. Let 
	\begin{align*}
	u &= \vf_m(x)^{n+1}\\
	dv &= \vf_m(x)^{m-1}\,dx,
	\end{align*}
	then Proposition \ref{diffEq}(2) implies
	\begin{align*}
	\hspace{.58in}du &= (n+1)\vf_m(x)^n\vf_m'(x)\,dx\\
	v &= -\tfrac{2}{m}\vf_m'(x).
	\end{align*}
	Thus,
	\begin{align}
	\label{intEq}
		I_m(n+m) &= \int_0^{\vp_m}{\vf_m(x)^{n+m}\,dx} \notag\\
		&= \left.-\tfrac{2}{m}\vf_m(x)^{n+1}\vf_m'(x)\right|_0^{\vp_m} + \tfrac{2(n+1)}{m}\int_0^{\vp_m}{\vf_m(x)^n\vf_m'(x)^2\,dx}\notag \\
		&= \tfrac{2(n+1)}{m}\int_0^{\vp_m}{\vf_m(x)^n\vf_m'(x)^2\,dx}\hspace{.64in}\text{Table 1} \notag \\
		&= \tfrac{2(n+1)}{m}\int_0^{\vp_m}{\vf_m(x)^n(1-\vf_m(x)^m)\,dx}\hspace{.25in}\text{Prop. } \ref{diffEq}(1)\notag \\
		&= \tfrac{2(n+1)}{m}\left(\int_0^{\vp_m}{\vf_m(x)^n\,dx} - \int_0^{\vp_m}{\vf_m(x)^{n+m}\,dx}\right)\notag \\
		&= \tfrac{2(n+1)}{m}(I_m(n) - I_m(n+m))
	\end{align}
	Rearranging \eqref{intEq} leads to \eqref{intRecur},
	\[
		\frac{I_m(n+m)}{I_m(n)} = \frac{2(n+1)}{2(n+1)+m}.\qedhere
	\]
\end{proof}

\begin{table}[t]
\label{table1}
	\centering
	\begin{tabular}{|c|c|c|}
		\hline
		$x$ & $\vf_m(x)$\rule{0pt}{11pt} & $\vf_m'(x)$ \\[2pt]
		\hline 
		$0$ & $0$\rule{0pt}{11pt} & $1$\\[2pt]                 
		\hline
		$\tfrac{1}{2}\vp_m$\rule{0pt}{11pt} & $1$ & $0$\\[2pt]
		\hline
		$\hspace{6pt}\vp_m$ & $0$\rule{0pt}{11pt} & $-1\hspace{8pt}$\\[2pt]
		\hline
	\end{tabular}
\caption{Important Values of $\vf_m$ and $\vf_m'$.}
\end{table}

Our next lemma demonstrates that $\{I_m(n)\}$ is a decreasing sequence.

\begin{lemma}
\label{orderLemma}
For any natural number $m$, if $n_1\geq n_2\geq 0$ then $I_m(n_1) \leq I_m(n_2)$.
\end{lemma}

\begin{proof}
	For $x\in[0,\vp_m]$ we have $\vf_m(x)\in[0,1]$. Therefore, $n_1\geq n_2\geq 0$ implies that $\vf_m(x)^{n_1}\leq \vf_m(x)^{n_2}$, and consequently $I_m(n_1) \leq I_m(n_2)$.
\end{proof}

It follows immediately from Lemma \ref{orderLemma} that $I_m(n) \geq I_m(n+1) \geq I_m(n+m)$. Dividing through by $I_m(n)$,
\[
	1 \geq \frac{I_m(n+1)}{I_m(n)} \geq \frac{I_m(n+m)}{I_m(n)} \stackrel{\eqref{intRecur}}{=} \frac{2(n+1)}{2(n+1)+m}.
\]
As $n$ tends toward infinity, the squeeze theorem implies

\begin{equation}
\label{limitEq}
	\lim_{n\rightarrow\infty}{\frac{I_m(n)}{I_m(n+1)}}=1.
\end{equation}

Next we compute the limit \eqref{limitEq} in another way using initial values of $I_m(n)$ and the recurrence \eqref{intRecur}. When $n=0$,
\[
	I_m(0) = \int_0^{\vp_m}{1\,dx} = \vp_m.
\]
Hence $I_m(m) = \tfrac{2}{m+2}I_m(0) = \tfrac{2}{m+2}\vp_m$. When $n=m-1$, we use Proposition \ref{diffEq}(2) and Table 1 to compute 
\begin{align*}
	I_m(m-1) &= \int_0^{\vp_m}{\vf_m(x)^{m-1}\,dx}\\
	&= -\tfrac{2}{m}\int_0^{\vp_m}{\vf_m''(x)\,dx}\\
	&= \left.-\tfrac{2}{m}\vf_m'(x)\right|_0^{\vp_m}\\
	&= \tfrac{4}{m}.
\end{align*}

\begin{table}[h]
\label{table2}
	\centering
	\begin{tabular}{|c|c|c|c|}
		\hline
		$n$ & $0$\rule{0pt}{11pt} & $m-1$ & $m$ \\[2pt]
		\hline 
		$I_m(n)$ & $\vp_m$\rule{0pt}{11pt} & $\tfrac{4}{m}$ & $\tfrac{2}{m+2}\vp_m$\\[2pt]                 
		\hline
	\end{tabular}
\caption{Initial Values of $I_m(n)$.}
\end{table}

Combining the initial values listed in Table 2 and \eqref{intRecur} allow us to give explicit formulae for $I_m(mn)$ and $I_m(mn-1)$. The formulae include products over all natural numbers in a congruence class up to a specified bound, for which the author has found no suitable notation.\footnote{One candidate notation is the \emph{multifactorial}, denoted $n!^{(k)} = n\cdot (n-k)!^{(k)}$ when $n\geq k$ and $1$ otherwise. However, the exponent introduces more clutter and the notation does not allow us to emphasize a fixed modulus as clearly.} Thus, we introduce the \emph{congruence Gamma function}, defined recursively by
\begin{align}
	\G_m(1,0) &= 1, \label{case1}\\
	\G_m(0,k) &= 1,	\label{case2}\\
	\G_m(n+1,k) &= (mn + k)\G_m(n,k). \label{case3}
\end{align}
For our purposes we will always assume $m$ and $n$ to be natural numbers and $0\leq k < m$ an integer. To summarize, $\G_m(n,k)$ is the product over the congruence class of $k \bmod m$ between $1$ and $mn$. When $m=1$ we have $\G_1(n,0) = \G(n) = (n-1)!$, where $\G(n)$ is the usual gamma function. The characteristic identity $\G(x+1) = x\G(x)$ and a simple inductive argument lead to the expression
\begin{align}
\label{gamma}
	\G_m(n,k) = \frac{\G(n+\tfrac{k}{m})}{m^n\G(\tfrac{k}{m})}.
\end{align}

The following lemma establishes recursive formulae for $I_m(mn)$ and $I_m(mn-1)$ in terms of the congruence Gamma function.

\begin{lemma}
\label{ratioLemma}
	Let $m$ and $n$ be natural numbers, then
	\begin{enumerate}
		\item $I_m(mn) = \frac{\G_{2m}(n,2)}{\G_{2m}(n,m+2)}\vp_m$,\\
		\item $I_m(mn-1) = \frac{\G_{2}(n,0)}{\G_{2}(n,1)}\frac{4}{m}$.
	\end{enumerate}
\end{lemma}

\begin{proof}
	Both identities follow by induction on $n$. The base case $n=1$ is a consequence of the initial values listed in Table 2.
	\[
	\begin{array}{rlcccl}
		&(1)\;\; I_m(m) &=& \frac{2}{m+2}\vp_m &=& \frac{\G_{2m}(1,2)}{\G_{2m}(1,m+2)}\vp_m,\\
		&(2)\;\; I_m(m - 1) &=& \frac{4}{m} &=& \frac{\G_2(1,0)}{\G_2(1,1)}\frac{4}{m}.\rule{0pt}{15pt}
	\end{array}
	\]
	Assuming the inductive hypothesis,
	\[
	\begin{array}{rlclcl}
		&(1)\;\; I_m(m(n + 1)) &\stackrel{\eqref{intRecur}}{=}& \frac{2(mn+1)}{2(mn+1) + m}I_m(mn) &=& \frac{(2mn + 2)\G_{2m}(n,2)}{(2mn + m + 2)\G_{2m}(n,m+2)}\vp_m\\
		&&&&\stackrel{\eqref{case3}}{=}& \frac{\G_{2m}(n+1,2)}{\G_{2m}(n+1,m+2)}\vp_m.\rule{0pt}{15pt}\\
		&(2)\;\; I_m(m(n + 1) - 1 ) &\stackrel{\eqref{intRecur}}{=}& \frac{2mn}{2mn+m}I_m(mn-1) &=& \frac{2n\G_{2}(n,0)}{(2n+1)\G_{2}(n,1)}\frac{4}{m}\\
		&&&&\stackrel{\eqref{case3}}{=}& \frac{\G_{2}(n+1,0)}{\G_{2}(n+1,1)}\frac{4}{m}.\rule{0pt}{15pt}
	\end{array}
	\]
	Therefore the formulae hold for all $n$ by the principle of induction.
\end{proof}

\begin{theorem}
For all natural $m$,
\begin{equation}
\label{main}
	\vp_m = \lim_{n\rightarrow\infty}{\frac{\G_{2}(n,0)\G_{2m}(n,m+2)}{\G_{2}(n,1)\G_{2m}(n,2)}\frac{4}{m}}.
\end{equation}
\end{theorem}

\begin{proof}
	From Lemma \ref{ratioLemma} we have
	\[
		\frac{I_m(mn-1)}{I_m(mn)} = \frac{\G_{2}(n,0)\G_{2m}(n,m+2)}{\G_{2}(n,1)\G_{2m}(n,2)}\frac{4}{m\vp_m}.
	\]
	Taking the limit as $n\rightarrow\infty$ we have $\lim_{n\rightarrow\infty}{\tfrac{I_m(mn-1)}{I_m(mn)}} = 1$ by \eqref{limitEq}. Therefore,
	\begin{equation}
	\label{limit}
		1 = \lim_{n\rightarrow\infty}{\frac{I_m(mn-1)}{I_m(mn)}} = \lim_{n\rightarrow\infty}{\frac{\G_{2}(n,0)\G_{2m}(n,m+2)}{\G_{2}(n,1)\G_{2m}(n,2)}\frac{4}{m\vp_m}}.
	\end{equation}
	Multiplying \eqref{limit} by $\vp_m$ results in \eqref{main}.
\end{proof}

Unwrapping the definition of $\G_m(n,k)$ reveals the infinite product expression promised in the introduction,
\begin{equation*}
	\vp_m = \frac{2(m+2)}{m}\prod_{n=1}^\infty{\frac{2n(2mn+ m +2)}{(2n+1)(2mn+2)}}.
\end{equation*}

\section*{Acknowledegements}
The author would like to thank Keith Conrad for all his wonderful expositons, in particular \cite{stirling} which inspired this project; David A. Cox for his support and guidance; Daniel J. Velleman and Gregory Call for encouraging me to write this note.

\end{document}